\documentclass[twoside, 12pt]{article}
\usepackage[russian, english]{babel}
\usepackage{epsfig}
\usepackage{amssymb,amsmath,amsfonts,soul,amsthm,enumerate}
\usepackage{amsfonts}
\usepackage{amsmath}
\usepackage{graphicx}
\usepackage{amsthm}
\usepackage[T2A]{fontenc}
\usepackage{mathrsfs}
\usepackage[english]{babel}
 \newcommand{\fig}[2]{{\small Fig.~#1. #2}}
\newtheorem{theorem}{Theorem}[section]
\newtheorem{lemma}{Lemma}[section]
\newtheorem{definition}{Definition}
\newtheorem{remark}{Remark}

\numberwithin{equation}{section}
 \def\@evenhead{\vbox{\hbox to \textwidth{\thepage\hfil\sl\leftmark\strut}\hrule}}
 \def\@oddhead{\vbox{\hbox to \textwidth{\rightmark\hfill\thepage\strut}\hrule}}

\begin{document}
 \sloppy

\selectlanguage{english}

\centerline{\bf Curvilinear parallelogram identity and mean-value property for}
\centerline{\bf a semilinear hyperbolic equation of second-order}     

\vskip 0.3cm

\centerline{\bf V.~I.~Korzyuk, J.~V.~Rudzko}        

\markboth{\hfill{\footnotesize\rm   V.~I.~Korzyuk, J.~V.~Rudzko  }\hfill}
{\hfill{\footnotesize\sl  Curvilinear parallelogram identity and mean-value property for a semilinear hyperbolic equation}\hfill}
\vskip 0.3cm

\vskip 0.7 cm


\vskip 0.2cm


\vskip 0.2cm

\noindent {\bf Abstract.} In this paper, we discuss some of the important qualitative properties of solutions of second-order hyperbolic equations, whose coefficients of the terms involving the second-order derivatives are independent of the desired function and its derivatives. Solutions of these equations have a special property called the curvilinear parallelogram identity (or the mean-value property), which can be used to solve some initial-boundary value problems.

\section{\large Introduction}

The terms ``mean value theorem'', ``mean value property'', ``mean formula'', and ``mean value'' are quite common in mathematics (e.~g., real and complex analysis, probability theory, partial differential equations) and physics. But they may pertain to diverse phenomena.

In the theory of partial differential equations, mean value theorems for harmonic functions and solutions of various elliptic equations are best known. They include the classical mean value property for harmonic functions \cite{KorzuykBook} and the results obtained in works \cite{Ilin1, Ilin2, Ilin3, PolovinkinEll1} for more general elliptic equations and elliptic operators. Similar theorems are formulated for (hypoelliptic) parabolic equations \cite{Kuptsov1, Kuptsov2, Kuptsov3}.

Such facts can be established not only for elliptic and parabolic equations but also for hyperbolic ones. First of all, we should mention the classical Asgeirsson's mean value theorem~\cite{Courant, Hoermander1} the ultrahyperbolic differential equation, and the mean value theorem of Bitsadze and Nakhushev for the wave equation~\cite{Bitsadze1}. Spherical means can be used to solve initial-value problems as it's done in the work~\cite{John} for the wave equation and the Darboux equation. Using a symbolic approach~\cite{Pokrovskii1} several results \cite{Meshkov1, Meshkov2, Meshkov3, Meshkov4, Meshkov5, Meshkov6, Meshkov7, Meshkov8, Meshkov9, Meshkov10} connected with mean values of solutions of various differential equations were obtained in works of Polovinkin and Meshkov et al. It should also be said that in these works, the parallelogram identity (parallelogram rule) for the wave equation (which the authors call `difference mean-value formula') was generalized to the following cases: a (nonstrictly) hyperbolic equation with constant coefficients of third-order \cite{Meshkov1}, fourth-order \cite{Meshkov2}, higher-order~\cite{Meshkov10}, an equation with constant coefficients and with the operator represented by the product of the first order hyperbolic operators and the second-order elliptic operators \cite{Meshkov7}. These results can be used to obtain analytical and numerical solutions to differential equations as it was done in~\cite{KorzuykBook, Korzyuk1, Jokhadze1, Matus1, Matus2}. However, these results are given mainly for equations with constant coefficients because of the methods used (Fourier transform, search for accompanying distribution with compact support).

Moreover, the characteristic parallelogram of differential equations has some applications in hydrodynamics \cite{Lavrentiev}.

In this paper, we derive the identity of a curvilinear characteristic parallelogram for a general semilinear second-order hyperbolic equation using the method of characteristics~\cite{KorzuykBook}. This identity can be considered as the mean value theorem in some sense.

\section{\large Semilinear hyperbolic equation}

\indent In the domain $\Omega \subseteq \mathbb{R}^2$ of two independent variables $\mathbf{x} = (x_1, x_2) \in \Omega \subseteq \mathbb{R}^2$, consider the semilinear hyperbolic equation of second-order
\begin{equation}\label{2.1}
A u(x_1, x_2) = f(x_1, x_2, u(x_1, x_2), \partial_{x_1} u(x_1, x_2), \partial_{x_2} u(x_1, x_2)),
\end{equation}
where the operator $A$ is defined as
$$
A u(x_1, x_2): = a(x_1, x_2) \partial_{x_1}^2 u(x_1, x_2) + 2 b(x_1, x_2) \partial_{x_1} \partial_{x_2} u(x_1, x_2) + c(x_1, x_2) \partial_{x_2}^2 u(x_1, x_2),
$$
and is hyperbolic (this means $b^2(\mathbf{x}) - a(\mathbf{x}) c(\mathbf{x}) > 0$ for any $x \in \Omega$).

Eq.~\eqref{2.1} has two families of characteristics: $\gamma_1(x_1,x_2)$ and $\gamma_2(x_1,x_2)$, which are the first integrals of the ordinary differential equation \cite{KorzuykBook}
\begin{equation}\label{2.2}
a(\mathbf{x}) (\mathrm{d}x_2)^2 - 2 b(\mathbf{x}) \mathrm{d}x_1 \mathrm{d}x_2 + c(\mathbf{x}) (\mathrm{d}x_1)^2 = 0,
\end{equation}
and solutions of the equation of characteristics \cite{KorzuykBook}
\begin{equation}\label{2.3}
a\left(\frac{\partial \gamma_i}{\partial x_1}\right)^2 + 2 b \frac{\partial \gamma_i}{\partial x_1} \frac{\partial \gamma_i}{\partial x_2} + c \left(\frac{\partial \gamma_i}{\partial x_2}\right)^2 = 0,\quad i =1,2.
\end{equation}
It is well known~\cite{KorzuykBook} that Eq. (2.2), in general, can be decomposed into two equations
$$
a(\mathbf{x}) \mathrm{d}x_2 - (b(\mathbf{x}) \pm \sqrt{b^2(\mathbf{x}) - a(\mathbf{x}) c(\mathbf{x})}) \mathrm{d}x_1 = 0, \text{ if } a(\mathbf{x}) \neq 0,
$$
or
$$
c(\mathbf{x}) \mathrm{d}x_1 - (b(\mathbf{x}) \pm \sqrt{b^2(\mathbf{x}) - a(\mathbf{x}) c(\mathbf{x})}) \mathrm{d}x_2 = 0, \text{ if } c(\mathbf{x}) \neq 0,
$$
or
$$
\mathrm{d}x_1 \mathrm{d}x_2 = 0, \text{ if } a(\mathbf{x}) = c(\mathbf{x}) = 0.
$$

Therefore, we can assume that $\gamma_1$ and $\gamma_2$ are the first integrals of different differential equations and they are functionally independent since the Jacobian $\left|\dfrac{\partial(\gamma_1, \gamma_2)}{\partial(x_1, x_2)}\right|$ is nonzero~\cite{KorzuykBook}.

If the curves $\gamma_i$, $i=1,2$, have a parametric representation $(x_1^{(i)}(t),x_2^{(i)}(t))$, where $x_j^{(i)}$, $j=1,2$, are some twice continuously differentiable functions, then the equation holds~\cite{DiBenedetto}
$$
a\big(D x_2^{(i)}\big)^2 - 2 b D x_1^{(i)} D x_2^{(i)} + c \big(D x_1^{(i)}\big)^2 = 0,\quad i=1,2,
$$
where $D$ is the operator of the ordinary derivative.

\section{\large Curvilinear characteristic parallelogram}

\begin{definition}
{\rm Curvilinear characteristic parallelogram of the hyperbolic differential equation~(2.1) is a set $\Pi = \{\mathbf{x}\,|\,\gamma_1(\mathbf{x}) \in [l_1, l_2] \land \gamma_2(\mathbf{x}) \in [r_1, r_2] \}$, where $l_1$, $l_2$, $r_1$, $r_2$ are some real numbers and $\gamma_i$, $i=1,2$ are two different functionally independent characteristics. }
\end{definition}

\begin{remark}
{\rm Definition 1 is correct. It is known~\cite{Amelya} that any other first integral (2.2) has the form $q \circ \gamma_1$, where $q$ is some continuously differentiable function. If $\gamma_1(\mathbf{x}) \in [l_1, l_2]$, then, due to the continuity of $q$, $q(\gamma_1(\mathbf{x})) \in q([l_1,l_2])=[\widetilde{l}_1,\widetilde{l}_2]$. Thus, the curvilinear characteristic parallelogram does not depend on the considered characteristics.}
\end{remark}

\begin{center}
\includegraphics[width=0.81\textwidth]{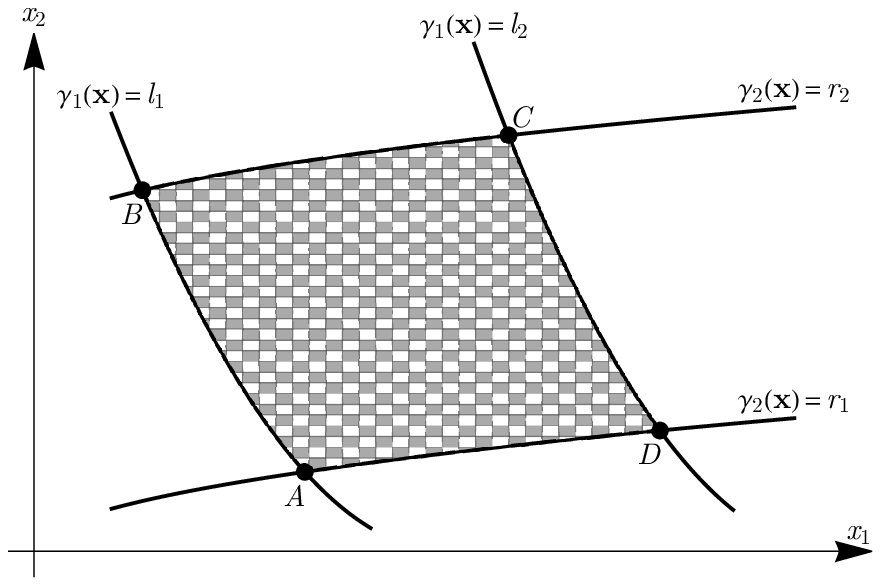}

\fig{1}{Curvilinear characteristic parallelogram}
\end{center}

\begin{definition}
{\rm Vertices of the curvilinear characteristic parallelogram $\Pi = \{\mathbf{x}\,|\,\gamma_1(\mathbf{x}) \in [l_1, l_2] \land \gamma_2(\mathbf{x}) \in [r_1, r_2] \}$ are points $\mathbf{x}$ such that $\gamma_1(x)=l_i \land \gamma_2(x)=r_j$, $(i,j)\in\{1,2\}\times\{1,2\}$. }
\end{definition}

\begin{remark}
{\rm Definition 2 is correct. We should show that $q \circ \gamma_1$, where $q$ is some continuously differentiable function, maps $[l_1, l_2]$ into $[\widetilde{l}_1,\widetilde{l}_2]$ and $\partial([l_1, l_2])$ into $\partial([\widetilde{l}_1,\widetilde{l}_2])$. If the function $q$ is increasing or decreasing, these mappings must be true. But if the function $q$ does not satisfy these conditions, then there exists at least one point $l_0 \in (l_1, l_2)$ such that $q'(l_0) = 0$. Due to the continuity of $q$, there exists a point $\mathbf{x} \in \Pi$ such that $\gamma_1(\mathbf{x}) = l_0 \in (l_1, l_2)$ This implies
$$
\left|\frac{\partial(q \circ \gamma_1,\gamma_2)}{\partial(x_1,x_2)}\right|(\mathbf{x}) = \begin{vmatrix}
q'(\gamma_1(\mathbf{x})) \partial_{x_1} \gamma_1(\mathbf{x}) & q'(\gamma_1(\mathbf{x})) \partial_{x_2} \gamma_1(\mathbf{x}) \\
\partial_{x_1} \gamma_2(\mathbf{x}) & \partial_{x_2} \gamma_2(\mathbf{x})
\end{vmatrix} = 0 \text{ when } \gamma_1(\mathbf{x}) = l_0.
$$
But we only consider characteristics with nonzero Jacobian. The correctness is proved.
}
\end{remark}

\begin{definition}
{\rm Opposite vertices of the curvilinear characteristic parallelogram $\Pi = \{\mathbf{x} \,|\, \gamma_1(\mathbf{x}) \in [l_1, l_2] \land \gamma_2(\mathbf{x}) \in [r_1, r_2] \}$ are its vertices $\mathbf{x}_1$ and $\mathbf{x}_2$ such that $\gamma_1(\mathbf{x}_1) \neq \gamma_1(\mathbf{x}_2)$ and $\gamma_2(\mathbf{x}_1) \neq \gamma_2(\mathbf{x}_2)$. }
\end{definition}

Point transformation of variables of the form $y_1=\gamma_1(x_1,x_2)$, $y_1=\gamma_2(x_1,x_2)$ is invertible \cite{Polyanin2005}, i.e., there is an inverse change of variables $x_1=\gamma_1^{-1}(y_1, y_2)$, $x_2=\gamma_2^{-1}(y_1,y_2)$.


\begin{lemma}\label{lem1}
Let $\Pi = \{\mathbf{x}\,|\,\gamma_1(\mathbf{x}) \in [l_1, l_2] \land \gamma_2(\mathbf{x}) \in [r_1, r_2] \}$ is a curvilinear characteristic parallelogram and the conditions $a \in C^2(\Pi)$, $b\in C^2(\Pi)$, $c\in C^2(\Pi)$, and $f\in C^1(\Pi\times\mathbb{R}^3)$ be satisfied. The function $u$ belongs to the class $C^2(\Pi)$ and satisfies Eq.~(2.1) if and only if it can be represented as
\begingroup
\allowdisplaybreaks
\begin{align*}
u(\mathbf{x})&=g_{1}\left(\gamma_{1}(\mathbf{x})\right)+g_{2}\left(\gamma_{2}(\mathbf{x})\right)+\\
&+\int\limits_{l^{(0)}}^{\gamma_{1}(\mathbf{x})} d z_{1} \int\limits_{r^{(0)}}^{\gamma_{2}(\mathbf{x})} \frac{1}{2\left(a \partial_{x_{1}} \gamma_{1} \partial_{x_{1}} \gamma_{2}+b\left(\partial_{x_{2}} \gamma_{2} \partial_{x_{1}} \gamma_{1}+\partial_{x_{2}} \gamma_{1} \partial_{x_{1}} \gamma_{2}\right)+c \partial_{x_{2}} \gamma_{1} \partial_{x_{2}} \gamma_{2}\right)(\gamma_{1}^{-1}(\mathbf{z}), \gamma_{2}^{-1}(\mathbf{z}))} \times \\
&\times \left[f\left(\gamma_{1}^{-1}(\mathbf{z}), \gamma_{2}^{-1}(\mathbf{z}), u\left(\gamma_{1}^{-1}(\mathbf{z}), \gamma_{2}^{-1}(\mathbf{z})\right), \right.\right. \\
&\left. \partial_{x_{1}} u\left(\gamma_{1}^{-1}(\mathbf{z}), \gamma_{2}^{-1}(\mathbf{z})\right), \partial_{x_{2}} u\left(\gamma_{1}^{-1}(\mathbf{z}), \gamma_{2}^{-1}(\mathbf{z})\right)\right) - \\
&-A \gamma_{1}\left(\gamma_{1}^{-1}(\mathbf{z}), \gamma_{2}^{-1}(\mathbf{z})\right)\left(\partial_{x_{1}} u\left(\gamma_{1}^{-1}(\mathbf{z}), \gamma_{2}^{-1}(\mathbf{z})\right) \partial_{y_{1}} \gamma_{1}^{-1}(\mathbf{z})\:+\right. \\
&\left.+\:\partial_{x_{2}} u\left(\gamma_{1}^{-1}(\mathbf{z}), \gamma_{2}^{-1}(\mathbf{z})\right) \partial_{y_{1}} \gamma_{2}^{-1}(\mathbf{z})\right) -\\
&-A \gamma_{2}\left(\gamma_{1}^{-1}(\mathbf{z}), \gamma_{2}^{-1}(\mathbf{z})\right)\left(\partial_{x_{1}} u\left(\gamma_{1}^{-1}(\mathbf{z}), \gamma_{2}^{-1}(\mathbf{z})\right) \partial_{y_{2}} \gamma_{1}^{-1}(\mathbf{z})\:+\right. \\
&\left.\left.+\:\partial_{x_{2}} u\left(\gamma_{1}^{-1}(\mathbf{z}), \gamma_{2}^{-1}(\mathbf{z})\right) \partial_{y_{2}} \gamma_{2}^{-1}(\mathbf{z})\right) \right] d z_{2}, \stepcounter{equation}\tag{\theequation}\label{3.1}
\end{align*}
\endgroup
where $l^{(0)} \in [l_1, l_2]$, $r^{(0)} \in [r_1, r_2]$, and the functions $g_1$, $g_2$ belong to the classes $C^2(\mathfrak{D}(g_1))$, $C^2(\mathfrak{D}(g_2))$ respectively.
\end{lemma}
\begin{proof}
Let the function $u \in C^2(\Pi)$ satisfy Eq. (2.1). Making the nonlinear nondegenerate change of independent variables $y_1=\gamma_1(x_1,x_2)$, $y_1=\gamma_2(x_1,x_2)$ and denoting $u(x_1,x_2 )=v(y_1,y_2)$ we obtain the new differential equation
$$
\begin{aligned}
&2\left(a \partial_{x_{1}} \gamma_{1} \partial_{x_{1}} \gamma_{2}+b\left(\partial_{x_{2}} \gamma_{2} \partial_{x_{1}} \gamma_{1}+\partial_{x_{2}} \gamma_{1} \partial_{x_{1}} \gamma_{2}\right)+c \partial_{x_{2}} \gamma_{1} \partial_{x_{2}} \gamma_{2}\right)\left(\gamma_{1}^{-1}(\mathbf{y}), \gamma_{2}^{-1}(\mathbf{y})\right) \times \\
&\times \partial_{y_{1}} \partial_{y_{2}} v(\mathbf{y}) + A \gamma_{1}\left(\gamma_{1}^{-1}(\mathbf{y}), \gamma_{2}^{-1}(\mathbf{y})\right) \partial_{y_{1}} v(\mathbf{y})+A \gamma_{2}\left(\gamma_{1}^{-1}(\mathbf{y}), \gamma_{2}^{-1}(\mathbf{y})\right) \partial_{y_{2}} v(\mathbf{y}) = \\
&=f\left(\gamma_{1}^{-1}(\mathbf{y}), \gamma_{2}^{-1}(\mathbf{y}), u\left(\gamma_{1}^{-1}(\mathbf{y}), \gamma_{2}^{-1}(\mathbf{y})\right), \partial_{x_{1}} u\left(\gamma_{1}^{-1}(\mathbf{y}), \gamma_{2}^{-1}(\mathbf{y})\right),\right.\\
&\left. \partial_{x_{2}} u\left(\gamma_{1}^{-1}(\mathbf{y}), \gamma_{2}^{-1}(\mathbf{y})\right)\right)= f\left(\gamma_{1}^{-1}(\mathbf{y}), \gamma_{2}^{-1}(\mathbf{y}), v(\mathbf{y}), \partial_{y_{1}} v(\mathbf{y}) \partial_{x_{1}} \gamma_{1}\left(\gamma_{1}^{-1}(\mathbf{y}), \gamma_{2}^{-1}(\mathbf{y})\right)\right.+ \\
&+\partial_{y_{2}} v(\mathbf{y}) \partial_{x_{1}} \gamma_{2}\left(\gamma_{1}^{-1}(\mathbf{y}), \gamma_{2}^{-1}(\mathbf{y})\right), \partial_{y_{1}} v(\mathbf{y}) \partial_{x_{2}} \gamma_{1}\left(\gamma_{1}^{-1}(\mathbf{y}), \gamma_{2}^{-1}(\mathbf{y})\right) \\
&\left.+\:\partial_{y_{2}} v(\mathbf{y}) \partial_{x_{2}} \gamma_{2}\left(\gamma_{1}^{-1}(\mathbf{y}), \gamma_{2}^{-1}(\mathbf{y})\right)\right)
\end{aligned}
$$
Let us integrate it twice to obtain the equation
$$
\begin{aligned}
v(\mathbf{y})&=g_{1}\left(\mathbf{y}\right)+g_{2}\left(\mathbf{y}\right)+\\
&+\int\limits_{l^{(0)}}^{y_1} d z_{1} \int\limits_{r^{(0)}}^{y_2} \frac{1}{2\left(a \partial_{x_{1}} \gamma_{1} \partial_{x_{1}} \gamma_{2}+b\left(\partial_{x_{2}} \gamma_{2} \partial_{x_{1}} \gamma_{1}+\partial_{x_{2}} \gamma_{1} \partial_{x_{1}} \gamma_{2}\right)+c \partial_{x_{2}} \gamma_{1} \partial_{x_{2}} \gamma_{2}\right)(\gamma_{1}^{-1}(\mathbf{y}), \gamma_{2}^{-1}(\mathbf{y}))} \times \\
&\times \left[f\left(\gamma_{1}^{-1}(\mathbf{z}), \gamma_{2}^{-1}(\mathbf{z}), u\left(\gamma_{1}^{-1}(\mathbf{z}), \gamma_{2}^{-1}(\mathbf{z})\right), \right.\right. \\
&\left. \partial_{x_{1}} u\left(\gamma_{1}^{-1}(\mathbf{z}), \gamma_{2}^{-1}(\mathbf{z})\right), \partial_{x_{2}} u\left(\gamma_{1}^{-1}(\mathbf{z}), \gamma_{2}^{-1}(\mathbf{z})\right)\right) - \\
&- \left. A \gamma_1 (\gamma_1^{-1}(\mathbf{z}), \gamma_2^{-1}(\mathbf{z})) \partial_{y_1} v(\mathbf{z}) - A \gamma_2(\gamma_1^{-1}(\mathbf{z}), \gamma_2^{-1}(\mathbf{z})) \partial_{y_2} v(\mathbf{z}) \right] d z_{2},
\end{aligned}
$$
Returning to the variables $x_1$ and $x_2$, we obtain Eq. (3.1). This also implies that the functions $g_j$ belong to the class $C^2(\mathfrak{D}(g_1))$, $j=1,2$. 

Substituting the representations (3.1) into Eq. (2.1), we verify that the function $u$ satisfies this equation in $\Pi$.
\end{proof}

\begin{remark}
{\rm Under some additional conditions on the functions $f$, $a$, $b$, $c$, $g_1$, $g_2$, we can show the solvability of the integro-differential equation (3.1) using the methods proposed in the works \cite{Evans, Korzyuk2, Vainberg}.
}
\end{remark}

For the convenience of further presentation, we introduce the notation
\begingroup
\allowdisplaybreaks
\begin{align*}
&\beta = 2\left(a \partial_{x_{1}} \gamma_{1} \partial_{x_{1}} \gamma_{2}+b\left(\partial_{x_{2}} \gamma_{2} \partial_{x_{1}} \gamma_{1}+\partial_{x_{2}} \gamma_{1} \partial_{x_{1}} \gamma_{2}\right)+c \partial_{x_{2}} \gamma_{1} \partial_{x_{2}} \gamma_{2}\right), \\
&K(\mathbf{z}, p, q, r) = f(\gamma_{1}^{-1}(\mathbf{z}), \gamma_{2}^{-1}(\mathbf{z}), p, q, r)\:- \\
&- A \gamma_{1}(\gamma_{1}^{-1}(\mathbf{z}), \gamma_{2}^{-1}(\mathbf{z})) (q \partial_{y_1} \gamma_1^{-1}(\mathbf{z}) + r \partial_{y_1} \gamma_2^{-1}(\mathbf{z}))\:- \\
&- A \gamma_{2}(\gamma_{1}^{-1}(\mathbf{z}), \gamma_{2}^{-1}(\mathbf{z})) (q \partial_{y_2} \gamma_1^{-1}(\mathbf{z}) + r \partial_{y_2} \gamma_2^{-1}(\mathbf{z})), \\
&\widetilde{K}(\mathbf{z}, p, q, r) = (\beta(\gamma_{1}^{-1}(\mathbf{z}), \gamma_{2}^{-1}(\mathbf{z})))^{-1} K(\mathbf{z}, p, q, r)
\end{align*}
\endgroup

\section{\large Curvilinear parallelogram identity}

\begin{theorem}\label{th1}
Let the conditions $a\in C^2(\Omega)$, $b\in C^2(\Omega)$, $c\in C^2(\Omega)$, $f\in C^1(\Omega\times \mathbb{R}^3)$, and $b^2(\mathbf{x}) - a(\mathbf{x}) c(\mathbf{x}) > 0$ be satisfied, and let the function $u$ belong to the class $C^2(\Omega)$ and be a solution of the hyperbolic equation~(2.1). Then for any curvilinear characteristic parallelogram $\Pi = \{\mathbf{x}\,|\,\gamma_1(\mathbf{x}) \in [l_1, l_2] \land \gamma_2(\mathbf{x}) \in [r_1, r_2] \} \subseteq \Omega$ with vertices $A(\gamma_1^{-1}(l_1,r_1),\gamma_2^{-1} (l_1,r_1))$, $B(\gamma_1^{-1}(l_1,r_2),\gamma_2^{-1} (l_1,r_2))$, $C(\gamma_1^{-1}(l_2,r_2),\gamma_2^{-1} (l_2,r_2))$, $(\gamma_1^{-1}(l_2,r_1),\gamma_2^{-1}(l_2,r_1))$, the equality holds
\begin{equation}\label{4.1}
\begin{aligned}
u(A) &- u(B) + u(C) - u(D) =\\
&= \int\limits_{l_1}^{l_2} d z_1 \int\limits_{r_1}^{r_2} \widetilde{K} \Big(\mathbf{z}, u\left(\gamma_{1}^{-1}(\mathbf{z}), \gamma_{2}^{-1}(\mathbf{z})\right), \partial_{x_{1}} u\left(\gamma_{1}^{-1}(\mathbf{z}), \gamma_{2}^{-1}(\mathbf{z})\right)\!,\\
&\quad\quad\partial_{x_{2}} u\left(\gamma_{1}^{-1}(\mathbf{z}), \gamma_{2}^{-1}(\mathbf{z})\right)\!\!\Big) d z_2.
\end{aligned}
\end{equation}
\end{theorem}

\begin{proof}
According to Lemma 1, the function $u$ can be represented in the form
\begin{equation}\label{4.2}
\begin{aligned}
u(\mathbf{x}) &= g_1(\gamma_1(\mathbf{x})) + g_2(\gamma_2(\mathbf{x}))\:+ \\
&+ \int\limits_{l_1}^{\gamma_1(\mathbf{x})} d z_1 \int\limits_{r_1}^{\gamma_2(\mathbf{x})} \widetilde{K} \Big(\mathbf{z}, u\left(\gamma_{1}^{-1}(\mathbf{z}), \gamma_{2}^{-1}(\mathbf{z})\right), \partial_{x_{1}} u\left(\gamma_{1}^{-1}(\mathbf{z}), \gamma_{2}^{-1}(\mathbf{z})\right)\!,\\
&\quad\quad\partial_{x_{2}} u\left(\gamma_{1}^{-1}(\mathbf{z}), \gamma_{2}^{-1}(\mathbf{z})\right)\!\!\Big) d z_2.
\end{aligned}
\end{equation}
where $g_i \in C^2(\mathfrak{D}(g_i))$, $i=1,2$. Using the expression (4.2), we compute
\begingroup
\allowdisplaybreaks
\begin{align*}
&u(A) = g_1(l_1) + g_2(r_1), u(B)=g_1(l_1)+g_2(r_2), u(D)=g_1(l_2)+g_2(r_1), \\
&u(C) = g_1(l_2) + g_2(r_2)\:+ \\
&+\int\limits_{l_1}^{l_2} d z_1 \int\limits_{r_1}^{r_2} \widetilde{K} \Big(\mathbf{z}, u\left(\gamma_{1}^{-1}(\mathbf{z}), \gamma_{2}^{-1}(\mathbf{z})\right), \partial_{x_{1}} u\left(\gamma_{1}^{-1}(\mathbf{z}), \gamma_{2}^{-1}(\mathbf{z})\right)\!,\\
&\quad\quad\partial_{x_{2}} u\left(\gamma_{1}^{-1}(\mathbf{z}), \gamma_{2}^{-1}(\mathbf{z})\right)\!\!\Big) d z_2. \stepcounter{equation}\tag{\theequation}\label{3.3}
\end{align*}
\endgroup
Substituting (4.3) into (4.1) gives the correct equality.
\end{proof}

\begin{theorem}\label{th2}
Let the conditions $u \in C^2(\Omega)$, $a\in C^2(\Omega)$, $b\in C^2(\Omega)$, $c\in C^2(\Omega)$, $f\in C^1(\Omega\times \mathbb{R}^3)$, and $b^2 (\mathbf{x})-a(\mathbf{x})c(\mathbf{x})>0$ be satisfied. 
If for any curvilinear characteristic parallelogram $\Pi = \{\mathbf{x}\,|\, \gamma_1(\mathbf{x}) \in [l_1, l_2] \land \gamma_2(\mathbf{x}) \in [r_1, r_2] \} \subseteq \Omega$ with vertices $A(\gamma_1^{-1}(l_1,r_1),\gamma_2^{-1} (l_1,r_1))$, $B(\gamma_1^{-1}(l_1,r_2),\gamma_2^{-1} (l_1,r_2))$, $C(\gamma_1^{-1}(l_2,r_2),\gamma_2^{-1} (l_2,r_2))$, $(\gamma_1^{-1}(l_2,r_1),\gamma_2^{-1}(l_2,r_1))$, where $\gamma_i$, $i=1,2$ are solutions of Eqs.~(2.2) and $\gamma_i^{-1}$ are defined as before, the equality~(4.1) is satisfied, then the function $u$ is a solution of Eq.~(2.1).
\end{theorem}

\begin{proof}
Let $l_2 = l + l_1$, $r_2 = r + r_1$. So, we can write the coordinates of points $A$, $B$, $C$ and $D$ in the form
$$
\begin{aligned}
A(\gamma_1^{-1}(l_1,r_1),\gamma_2^{-1} (l_1,r_1)),\:& B(\gamma_1^{-1}(l_1,r + r_1),\gamma_2^{-1} (l_1,r + r_1)), \\
C(\gamma_1^{-1}(l + l_1,r + r_1),\gamma_2^{-1} (l + l_1,r + r_1)),\:& D(\gamma_1^{-1}(l + l_1,r_1),\gamma_2^{-1}(l + l_1,r_1)).
\end{aligned}
$$
Let's consider the expression
$$
\begin{aligned}
\frac{u(A) - u(B)}{r} &= \frac{u(\gamma_1^{-1}(l_1,r_1),\gamma_2^{-1} (l_1,r_1)) - u(\gamma_1^{-1}(l_1,r + r_1),\gamma_2^{-1} (l_1,r + r_1))}{r} \xrightarrow[r \to 0]{\:} \\ 
&\xrightarrow[r \to 0]{\:} -\partial_r u(\gamma_1^{-1}(l_1,r_1),\gamma_2^{-1} (l_1,r_1)).
\end{aligned}
$$
In the same way
$$
\frac{u(C) - u(D)}{r} \xrightarrow[r \to 0]{\:} \partial_r u(\gamma_1^{-1}(l_1 + l,r_1),\gamma_2^{-1} (l_1 + l,r_1)).
$$
Now since
$$
\begin{aligned}
& \frac{\partial_r u(\gamma_1^{-1}(l_1 + l,r_1), \gamma_2^{-1} (l_1 + l,r_1)) - \partial_r u(\gamma_1^{-1}(l_1,r_1),\gamma_2^{-1} (l_1,r_1))}{l} \xrightarrow[l \to 0]{\:} \\
&\quad\quad\quad\quad\xrightarrow[l \to 0]{\:} \partial_l \partial_r u(\gamma_1^{-1}(l_1,r_1),\gamma_2^{-1} (l_1,r_1)),
\end{aligned}
$$ 
we obtain $\lim\limits_{(r, l) \to (0,0)} (l r)^{-1} (u(A) - u(B) + u(C) - u(D)) = \partial_l \partial_r u(\gamma_1^{-1}(l_1,r_1),\gamma_2^{-1} (l_1,r_1))$. 
Similarly, we get
$$
\begin{aligned}
&\lim\limits_{(r, l) \to (0,0)} \frac{1}{l r} \int\limits_{l_1}^{l + l_1} d z_1 \int\limits_{r_1}^{r + r_1} \widetilde{K} \Big(\mathbf{z}, u\left(\gamma_{1}^{-1}(\mathbf{z}), \gamma_{2}^{-1}(\mathbf{z})\right), \partial_{x_{1}} u\left(\gamma_{1}^{-1}(\mathbf{z}), \gamma_{2}^{-1}(\mathbf{z})\right)\!,\\
&\partial_{x_{2}} u\left(\gamma_{1}^{-1}(\mathbf{z}), \gamma_{2}^{-1}(\mathbf{z})\right)\!\!\Big) d z_2 = \\
&= \widetilde{K}\Big(\mathbf{z} = (l_1, r_1), u\left(\gamma_{1}^{-1}(\mathbf{z}), \gamma_{2}^{-1}(\mathbf{z})\right), \partial_{x_{1}} u\left(\gamma_{1}^{-1}(\mathbf{z}), \gamma_{2}^{-1}(\mathbf{z})\right), \partial_{x_{2}} u\left(\gamma_{1}^{-1}(\mathbf{z}), \gamma_{2}^{-1}(\mathbf{z})\right)\!\!\Big).
\end{aligned}
$$

Thus
\begingroup
\allowdisplaybreaks
\begin{align*}
&\lim\limits_{(r, l) \to (0,0)} \frac{1}{l r} \Bigg(u(A) - u(B) + u(C) - u(D) - \\
&- \int\limits_{l_1}^{l + l_1} d z_1 \int\limits_{r_1}^{r + r_1} \widetilde{K}\Big(\mathbf{z}, u\left(\gamma_{1}^{-1}(\mathbf{z}), \gamma_{2}^{-1}(\mathbf{z})\right), \partial_{x_{1}} u\left(\gamma_{1}^{-1}(\mathbf{z}), \gamma_{2}^{-1}(\mathbf{z})\right)\!, \partial_{x_{2}} u\left(\gamma_{1}^{-1}(\mathbf{z})\!, \gamma_{2}^{-1}(\mathbf{z})\right)\!\!\Big) d z_2 \Bigg) = \\
& = \lim\limits_{(r, l) \to (0,0)} \frac{u(A) - u(B) + u(C) - u(D)}{l r} - \\
& - \lim\limits_{(r, l) \to (0,0)} \frac{1}{l r} \int\limits_{l_1}^{l + l_1} d z_1 \int\limits_{r_1}^{r + r_1} \widetilde{K} \Big(\mathbf{z}, u\left(\gamma_{1}^{-1}(\mathbf{z}), \gamma_{2}^{-1}(\mathbf{z})\right), \partial_{x_{1}} u\left(\gamma_{1}^{-1}(\mathbf{z}), \gamma_{2}^{-1}(\mathbf{z})\right)\!, \\
&\partial_{x_{2}} u\left(\gamma_{1}^{-1}(\mathbf{z})\!, \gamma_{2}^{-1}(\mathbf{z})\right)\!\!\Big) d z_2 = \partial_l \partial_r u(\gamma_1^{-1}(l_1,r_1),\gamma_2^{-1} (l_1,r_1)) - \\
&-\:\frac{K\left(\mathbf{z} = (l_1, r_1), u\left(\gamma_{1}^{-1}(\mathbf{z}), \gamma_{2}^{-1}(\mathbf{z})\right), \partial_{x_{1}} u\left(\gamma_{1}^{-1}(\mathbf{z}), \gamma_{2}^{-1}(\mathbf{z})\right)\!, \partial_{x_{2}} u\left(\gamma_{1}^{-1}(\mathbf{z})\!, \gamma_{2}^{-1}(\mathbf{z})\right)\right)}{\beta(\gamma_1^{-1}(l_1,r_1),\gamma_2^{-1} (l_1,r_1))}.
\end{align*}
\endgroup
This means that the function $u$ satisfies at the point ${(\gamma_1^{-1}(\mathbf{z}=(y_1=l_1,y_2=r_1)),\gamma_2^{-1}(\mathbf{z}))}$ the differential equation
\begin{equation}\label{4.4}
\begin{aligned}
&\beta(\gamma_1^{-1}(\mathbf{z}),\gamma_2^{-1}(\mathbf{z})) \partial_{y_1} \partial_{y_2} u(\gamma_1^{-1}(\mathbf{z}),\gamma_2^{-1}(\mathbf{z})) = \\
&= f\left(\gamma_{1}^{-1}(\mathbf{z}), \gamma_{2}^{-1}(\mathbf{z}), u\left(\gamma_{1}^{-1}(\mathbf{z}), \gamma_{2}^{-1}(\mathbf{z})\right), \right. \\
&\left. \partial_{x_{1}} u\left(\gamma_{1}^{-1}(\mathbf{z}), \gamma_{2}^{-1}(\mathbf{z})\right), \partial_{x_{2}} u\left(\gamma_{1}^{-1}(\mathbf{z}), \gamma_{2}^{-1}(\mathbf{z})\right)\right) - \\
&-A \gamma_{1}\left(\gamma_{1}^{-1}(\mathbf{z}), \gamma_{2}^{-1}(\mathbf{z})\right)\left(\partial_{x_{1}} u\left(\gamma_{1}^{-1}(\mathbf{z}), \gamma_{2}^{-1}(\mathbf{z})\right) \partial_{y_{1}} \gamma_{1}^{-1}(\mathbf{z})\:+\right. \\
&\left.+\:\partial_{x_{2}} u\left(\gamma_{1}^{-1}(\mathbf{z}), \gamma_{2}^{-1}(\mathbf{z})\right) \partial_{y_{1}} \gamma_{2}^{-1}(\mathbf{z})\right) -\\
&-A \gamma_{2}\left(\gamma_{1}^{-1}(\mathbf{z}), \gamma_{2}^{-1}(\mathbf{z})\right)\left(\partial_{x_{1}} u\left(\gamma_{1}^{-1}(\mathbf{z}), \gamma_{2}^{-1}(\mathbf{z})\right) \partial_{y_{2}} \gamma_{1}^{-1}(\mathbf{z})\:+\right. \\
&\left.+\:\partial_{x_{2}} u\left(\gamma_{1}^{-1}(\mathbf{z}), \gamma_{2}^{-1}(\mathbf{z})\right) \partial_{y_{2}} \gamma_{2}^{-1}(\mathbf{z})\right),
\end{aligned}
\end{equation}
where $x_1=\gamma_1^{-1}(y_1, y_2)$, $x_2=\gamma_2^{-1}(y_1, y_2)$. By virtue of the arbitrariness of $\Pi \subseteq \Omega$, Eq. (4.4) is true for any point $(x_1=\gamma_1^{-1}(\mathbf{z}=(l_1,r_1)),x_2=\gamma_2^{-1} (\mathbf{z}=(l_1,r_1))) \in \Omega$. 

Making the change of variables $x_1=\gamma_1^{-1}(y_1, y_2)$, $x_2=\gamma_2^{-1}(y_1, y_2)$ in Eq. (4.4), we obtain Eq. (\ref{2.1}).
\end{proof}

Note that formula (4.1) can be considered as a kind of mean value theorem.

\section{\large Applications}

\subsection{Wave equation}

Let's consider $A u (x_1, x_2) = \partial_{x_1}^2 u (x_1, x_2) - a^2 \partial_{x_2}^2 u (x_1, x_2)$, where $a>0$ (for definiteness). Then we have $\gamma_1(x_1, x_2) = x_2 - a x_1$, $\gamma_2(x_1, x_2) = x_2 + a x_1$, $\gamma^{-1}_1(y_1, y_2) = (y_2 - y_1)/(2 a)$, $\gamma^{-1}_2(y_1, y_2) = (y_1 + y_2)/2$, $A \gamma_1\equiv 0$, $A \gamma_2\equiv 0$. 

\subsubsection{Parallelogram identity}

Let $f \equiv 0$. In this case, formula (4.1) transforms to
\begingroup
\allowdisplaybreaks
\begin{align*}
&u \left(\frac{r_1-l_1}{2 a}, \frac{l_1+r_1}{2}\right) - u\left(\frac{r_2-l_1}{2 a},\frac{l_1+r_2}{2}\right) + \\
+ \: &u\left(\frac{r_2-l_2}{2 a}, \frac{l_2+r_2}{2}\right) - u\left(\frac{r_1-l_2}{2 a}, \frac{l_2+r_1}{2}\right) = 0,
\stepcounter{equation}\tag{\theequation}\label{5.1}
\end{align*}
\endgroup
where $l_1$, $l_2$, $r_1$ and $r_2$ are some real numbers. Eq. (5.1) is the well-known parallelogram identity for the wave equation.

\subsubsection{Goursat problem}

Let's consider the Goursat problem \cite{Kovnatskaya1}
\begin{equation}\label{5.2}
\left\{\begin{array}{l}
(\partial_{x_1}^2 - a^2 \partial_{x_2}^2) u (\mathbf{x}) = f(\mathbf{x}), \quad 0 < x_1, -a x_1 < x_2 < a x_1, \\
u(x_1, x_2 = a x_1) = \phi^{(1)}(x_1), \quad u (x_1, x_2 = -a x_1) = \phi^{(1)}(x_2),\quad x_1 > 0,
\end{array}\right.
\end{equation}
where $f \in C^1 (\{\mathbf{x}\,|\,0 \leqslant x_1, -a x_1 \leqslant x_2 \leqslant a x_1\})$, $\phi^{(1)} \in C^2([0,\infty))$, $\phi^{(2)} \in C^2([0,\infty))$ and $\phi^{(1)}(0) = \phi^{(2)}(0)$. We can write the classical solution of (5.2) using the formula (4.1). If we choose $C(x_1, x_2)$, $B\!\left(\dfrac{a x_1+x_2}{2 a}, \dfrac{a x_1+x_2}{2}\right)$, $D\!\left(\dfrac{a x_1-x_2}{2 a}, \dfrac{x_2 - a x_1}{2}\right)$, $A(0, 0)$ and apply (4.1), we obtain
\begingroup
\allowdisplaybreaks
\begin{align*}
u(x_1, &x_2) = u(C) = \phi^{(1)}\!\left(\frac{a x_1+x_2}{2 a}\right) + \phi^{(2)}\!\left(\frac{a x_1-x_2}{2 a}\right) - \phi^{(1)}(0)\: - \\ 
&- \frac{1}{4 a^2} \int\limits_{0}^{x_2 - a x_1} d y_1 \int\limits_{0}^{x_2 + a x_1} f\left(\frac{y_2-y_1}{2 a}, \frac{y_1+y_2}{2}\right) d y_2, \quad 0 < x_1, -a x_1 < x_2 < a x_1.
\end{align*}
\begingroup
\begin{center}
\includegraphics[width=0.81\textwidth]{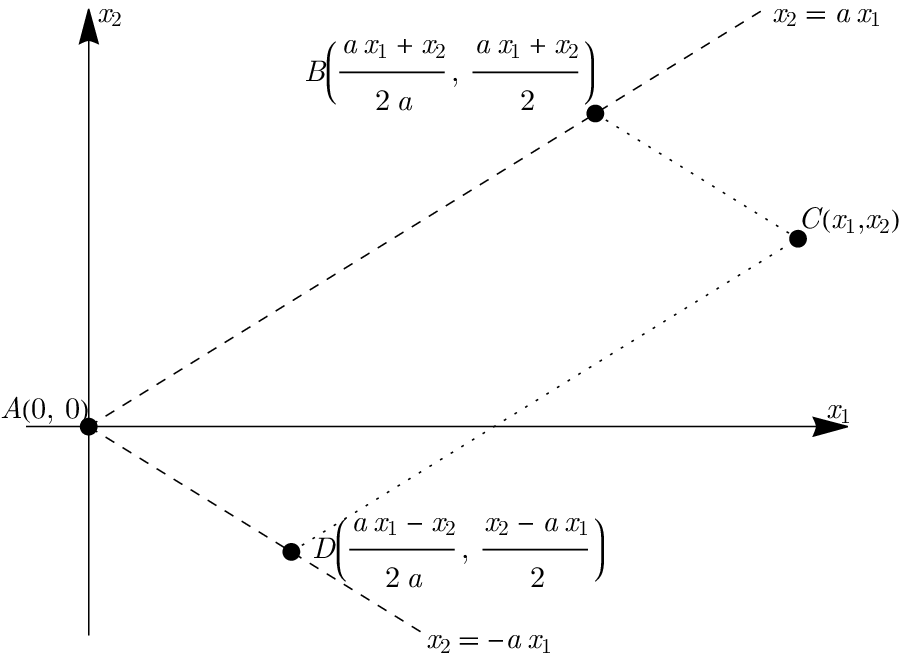}

\fig{2}{To the Goursat problem (5.2).}
\end{center}

\subsubsection{Mixed problem}

Let's consider the first mixed problem \cite{KorzuykBook}
\begin{equation}\label{5.3}
\left\{\begin{array}{l}
(\partial_{x_1}^2 - a^2 \partial_{x_2}^2) u (\mathbf{x}) = f(\mathbf{x}), \quad \mathbf{x} \in (0, \infty)\times(0, \infty), \\
u(0, x_2) = \phi(x_1), \quad \partial_{x_1} u (0,x_2) = \psi(x_2), \quad x_1 > 0, \\
u(x_1, 0) = \mu(x_1), \quad x_2 > 0,
\end{array}\right.
\end{equation}
where $f \in C^1 ([0, \infty)\times[0, \infty))$, $\phi \in C^2([0,\infty))$, $\psi \in C^1([0,\infty))$, $\mu \in C^2([0,\infty))$.

If $x_2 - a x_1 > 0$, then the solution of (5.3) at the point $(x_1, x_2)$ can be defined by d'Alembert formula
\begin{equation}\label{5.4}
\begin{aligned}
u(x_1, x_2) &= \frac{\phi(x_2 - a x_1) + \phi(x_2 + a x_1)}{2} + \frac{1}{2 a}\int\limits_{x_2 - a x_1}^{x_2 + a x_1}\psi(\xi)\,d\xi \:+ \\
&+ \frac {1}{2 a}\int\limits_{0}^{x_1} d\tau \int\limits_{x_2 - a(x_1 - \tau)}^{x_2 + a(x_1 - \tau)} f (\tau, \xi) d\xi, \quad x_2 - a x_1 > 0,\quad x_1 > 0,\quad x_2 > 0.
\end{aligned}
\end{equation}

If $x_2 - a x_1 < 0$, then we can use the parallelogram identity (4.1) to derive the solution of (5.3) at the point $(x_1, x_2)$. We can choose $C(x_1, x_2)$, $B\!\left(x_1-\dfrac{x_2}{a},0\right)$, $D\!\left(\dfrac{x_2}{a},a x_1\right)$, $A\left(0,a x_1-x_2\right)$, apply (4.1) and obtain

\begin{equation}\label{5.5}
\begin{aligned}
u(x_1, x_2) &= \mu\left(x_1-\frac{x_2}{a}\right)+\frac{\phi \left(a x_1+x_2\right)-\phi \left(a x_1-x_2\right)}{2} + \frac{1}{2 a} \int\limits_{a x_1-x_2}^{a x_1+x_2} \psi (\xi ) \, d\xi \:+ \\
&+ \frac {1}{2 a}\int\limits_{0}^{\frac{x_2}{a}} d\tau \int\limits_{a x_1 - x_2 + a \tau}^{a x_1 + x_2 - a \tau} f (\tau, \xi) d\xi - \frac{1}{4 a^2} \int\limits_{a x_1-x_2}^{x_2-a x_1} d y_1 \int\limits_{a x_1-x_2}^{a x_1+x_2} f\left(\frac{y_2 - y_1}{2 a},\frac{y_2 + y_1}{2}\right) \,
   d y_2, \\
& x_2 - a x_1 < 0, \quad x_1 > 0, \quad x_2 > 0.
\end{aligned}
\end{equation}
\begin{center}
\includegraphics[width=0.81\textwidth]{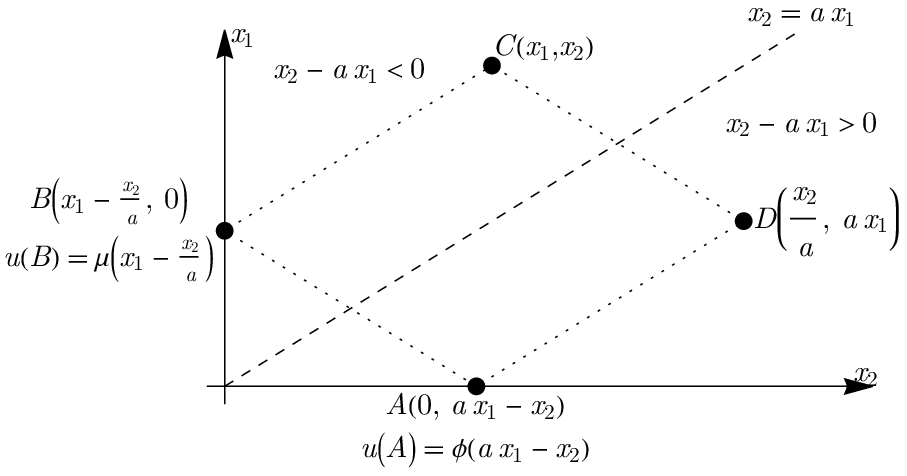}

\fig{3}{To the first mixed problem (5.3).}
\end{center}
Using the representations (5.4) and (5.5), we can easily derive the necessary and sufficient matching conditions $\mu(0) = \phi(0)$, $\mu'(0) = \psi(0)$ and $\mu''(0) = a^2 \phi''(0) + f(0,0)$ under which the solution $u$ of the first mixed problem (5.3) will be classical.

\subsection{Nonlinear wave equation}

For convenience, we will present equations in divergence form later in this section. Let's consider $A u (x_1, x_2) = \partial_{x_1} \partial_{x_2} u (x_1, x_2)$. Then we have $\gamma_1(x_1, x_2) = x_1$, $\gamma_2(x_1, x_2) = x_2$, $\gamma^{-1}_1(y_1, y_2) = y_1$, $\gamma^{-1}_2(y_1, y_2) = y_2$, $A \gamma_1\equiv 0$, $A \gamma_2\equiv 0$.

\subsubsection{Darboux problem}

Let's consider the second Darboux problem for a nonlinear wave equation in divergence form \cite{Jokhadze1}

\begin{equation}\label{5.6}
\left\{\begin{array}{l}
\partial_{x_1} \partial_{x_2} u (\mathbf{x}) + \lambda g(\mathbf{x}, u(\mathbf{x})) = f(\mathbf{x}), \quad 0 < x_1, \alpha x_1 < x_2 < \beta x_1, \\
u(x_1, x_2 = \alpha x_1) = u(x_1, x_2 = \beta x_1) = 0, \quad x_1 > 0,
\end{array}\right.
\end{equation}
where $\lambda \in \mathbb{R}$, $0 < \alpha < 1 < \beta < \infty$, $f \in C^1(\{\mathbf{x}\,|\,0 \leqslant x_1 \land \alpha x_1 \leqslant x_2 \leqslant \beta x_1\})$, $g \in C^1(\{\mathbf{x}\,|\,0 \leqslant x_1 \land \alpha x_1 \leqslant x_2 \leqslant \beta x_1\} \times \mathbb{R})$, $|g(x_1, x_2, z)| \leqslant L_1 + L_2 |z|$, $L_1 \geqslant 0$, $L_2 \geqslant 0$.

We want to obtain an expression for the classical solution $u$ of the problem (5.6) at the point $P_0(x_1, x_2)$. Let us denote by $P_1 M_0 P_0 N_0$ the characteristic parallelogram, whose vertices $N_0$ and $M_0$ lie, respectively, on the segments $x_2 = \alpha x_1$ and $x_2 = \beta x_1$, that is: $N_0: = (x_1, \alpha x_1)$, $M_0: = \left(\beta^{-1} x_2, x_2\right)$, $P_1: = \left(\beta^{-1} x_2, \alpha x_1\right)$. Since $P_1 \in \{\mathbf{x}\,|\,0 < x_1 \land \alpha x_1 < x_2 < \beta x_1\}$, we construct analogously the characteristic parallelogram $P_2 M_1 P_1 N_1$ whose vertices $N_1$ and $M_1$ lie, respectively, on the segments $x_2 = \alpha x_1$ and $x_2 = \beta x_1$. Continuing this process, we obtain the characteristic parallelogram $P_{i+1} M_i P_i N_i$ for which $N_i \in \{\mathbf{x}\,|\, x_2 = \alpha x_1 \}$, $M_i \in \{\mathbf{x}\,|\, x_2 = \beta x_1 \}$, and $N_i: = \left(x_1^{(i)}, \alpha x_1^{(i)}\right)$, $M_i: = \left(\beta^{-1} x_2^{(i)}, x_2\right)$, $P_{i+1}: = \left(\beta^{-1} x_2^{(i)}, \alpha x_1^{(i)}\right)$ if $P_i: = \left(x_1^{(i)}, x_2^{(i)}\right)$.

%
\begin{figure}[h]
\centering
\includegraphics[width=0.81\textwidth]{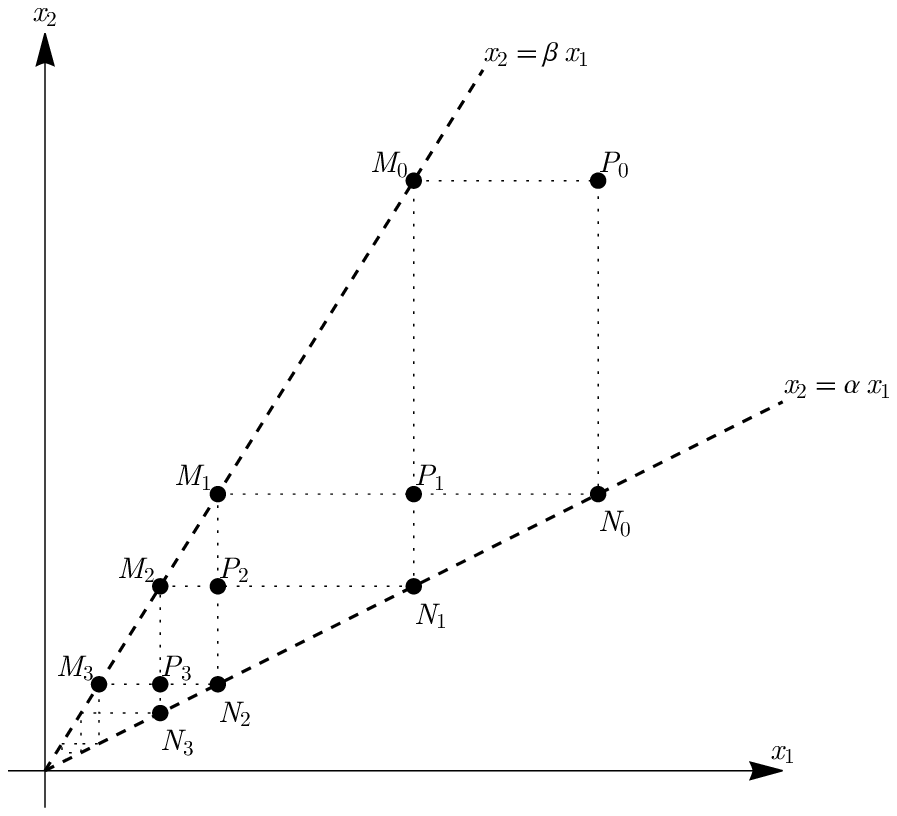}

\fig{4}{To the second Darboux problem (5.6).}
\end{figure}


By (4.1) and (5.6), we have
\begingroup
\allowdisplaybreaks
\begin{align*}
u(P_i) = u(M_i) + u(N_i) - u(P_{i+1}) &+ \iint\limits_{P_{i+1} M_i P_i N_i} \left[f(\mathbf{z}) - \lambda g(\mathbf{z}, u(\mathbf{z}))\right] d\mathbf{z} = \\
= -u(P_{i+1}) &+ \iint\limits_{P_{i+1} M_i P_i N_i} \left[f(\mathbf{z}) - \lambda g(\mathbf{z}, u(\mathbf{z}))\right] d\mathbf{z},\:\:i \in \mathbb{N} \cup \{ 0 \}.
\end{align*}
\endgroup
It follows that
\begingroup
\allowdisplaybreaks
\begin{align*}
u(x_1, x_2) &= u(P_0) = \iint\limits_{P_{1} M_0 P_0 N_0} \left[f(\mathbf{z}) - \lambda g(\mathbf{z}, u(\mathbf{z}))\right] d\mathbf{z} - u(P_{1}) = \\
&= u(P_2) + \iint\limits_{P_{1} M_0 P_0 N_0} \left[f(\mathbf{z}) - \lambda g(\mathbf{z}, u(\mathbf{z}))\right] d\mathbf{z} - \iint\limits_{P_{2} M_1 P_1 N_1} \left[f(\mathbf{z}) - \lambda g(\mathbf{z}, u(\mathbf{z}))\right] d\mathbf{z} = \\
&= (-1)^n u(P_n) + \sum\limits_{i=0}^{n-1} (-1)^i \iint\limits_{P_{i+1} M_i P_i N_i} \left[f(\mathbf{z}) - \lambda g(\mathbf{z}, u(\mathbf{z}))\right] d\mathbf{z}.
\end{align*}
\endgroup
It is clear that $\lim\limits_{n \rightarrow \infty} u(P_n) = u\left(\lim\limits_{n \rightarrow \infty} P_n\right) = u(0, 0) = 0$. Hence, passing to the limit, as $n \rightarrow \infty$, we obtain the following integral representation
\begin{equation}\label{5.7}
u(x_1, x_2) = \sum\limits_{i=0}^{\infty} (-1)^i \iint\limits_{P_{i+1} M_i P_i N_i} \left[f(\mathbf{z}) - \lambda g(\mathbf{z}, u(\mathbf{z}))\right] d\mathbf{z}.
\end{equation}
The further solution of the problem (5.6) is connected with the study of the solvability of Eq. (5.7), and it is given in the work \cite{Jokhadze1}. It turns out that under the conditions given in the statement of the problem (5.6), it has a unique classical solution. But we still notice that in the linear case, i.e., when $\lambda = 0$, the formula (5.7) transforms into
\begin{equation}\label{5.8}
u(x_1, x_2) = \sum\limits_{i=0}^{\infty} (-1)^i \iint\limits_{P_{i+1} M_i P_i N_i} f(\mathbf{z}) d\mathbf{z},
\end{equation}
The series in the right-hand side of equality (5.8) is uniformly and absolutely convergent \cite{Jokhadze1}. Thus, in the linear case, there exists the solution $u$ of (5.6) written in the explicit analytic form (5.8).

\subsection{Linear second-order hyperbolic equation}

As in the previous subsection, we consider $A u (x_1, x_2) = \partial_{x_1} \partial_{x_2} u (x_1, x_2)$. Then we have $\gamma_1(x_1, x_2) = x_1$, $\gamma_2(x_1, x_2) = x_2$, $\gamma^{-1}_1(y_1, y_2) = y_1$, $\gamma^{-1}_2(y_1, y_2) = y_2$, $A \gamma_1\equiv 0$, $A \gamma_2\equiv 0$.

\subsubsection{Goursat problem}

Let's consider the Goursat problem for a linear second-order hyperbolic equation \cite{KorzuykBook}
\begin{equation}\label{5.9}
\left\{\begin{array}{l}
\partial_{x_1} \partial_{x_2} u (\mathbf{x}) + a(\mathbf{x}) \partial_{x_1} u (\mathbf{x}) + b(\mathbf{x}) \partial_{x_2} u (\mathbf{x}) + c(\mathbf{x}) u (\mathbf{x}) = f(\mathbf{x}), \quad x_1^{(0)} < x_1, x_2^{(0)} < x_2, \\
u(x_1 = x_1^{(0)}, x_2) = \phi(x_2), \quad x_2 > x_2^{(0)}, \\
u(x_1, x_2 = x_2^{(0)}) = \psi(x_1), \quad x_1 > x_1^{(0)},
\end{array}\right.
\end{equation}
where $f \in C(\{\mathbf{x}\,|\,x_1^{(0)} \leqslant x_1 \land x_2^{(0)} \leqslant x_2\})$, $\phi \in C^1([x_2^{(0)},\infty))$, $\psi \in C^1([x_1^{(0)},\infty))$ and $\phi(x_2^{(0)}) = \psi(x_1^{(0)})$. We can write the classical solution of (5.9) using the formula (4.1). If we choose $C(x_1, x_2)$, $B(x_1^{(0)}, x_2)$, $D(x_1, x_2^{(0)})$, $A(x_1^{(0)}, x_2^{(0)})$ and apply (4.1), we obtain

\begin{align*}
u (\mathbf{x}) &= u(C) = \phi(x_2) + \psi(x_1) - \psi(x_2^{(0)})\:+ \\
& + \int\limits_{x_1^{(0)}}^{x_1} d y_1 \int\limits_{x_2^{(0)}}^{x_2} [f(\mathbf{y}) - a(\mathbf{y}) \partial_{x_1} u (\mathbf{y}) - b(\mathbf{y}) \partial_{x_2} u (\mathbf{y}) - c(\mathbf{y}) u (\mathbf{y})] d y_2. \stepcounter{equation}\tag{\theequation}\label{5.10}
\end{align*}

A representation of the solution in the form of the integro-differential equation (5.10) is obtained. Under the conditions specified in the formulation of the problem (5.9), Eq. (5.10) will be solvable \cite{KorzuykBook}, and the function $u$ will have the required smoothness. It proves the solvability of the problem (5.9).

\section{\large Conclusion}

In this paper, we have generalized the parallelogram rule for the wave equation to the case of a semilinear hyperbolic equation of the second order. This identity connects not only the values of the points at the vertices of the parallelogram but also the continuum of function values on the parallelogram, in contrast to the linear cases with constant coefficients considered earlier. We have shown how the obtained results can be used in combination with other methods to solve various mixed problems.

%
%

\def\bibname{\vspace*{-30mm}{

\vskip 1 cm \footnotesize
\begin{flushleft}
Viktor~Korzyuk, Jan~Rudzko \\ 
Department of Mathematical Cybernetics \\ 
Belarusian State University \\ 
4 Nezavisimosti Avenue,\\ 
Minsk, Belarus \\ 
E-mails: korzyuk@bsu.by, janycz@yahoo.com 
\end{flushleft}


\vskip0.5cm
\begin{flushright}
Received: ??.??.20??
\end{flushright}

 \end{document}